\documentclass[11pt]{article}

\usepackage{amsfonts}
\usepackage{amsthm}
\usepackage{amsmath}
\usepackage{graphicx}
\usepackage{enumerate}

\newtheorem{theorem}{Theorem}

\newtheorem{definition}{Definition}
\newtheorem{example}{Example}
\newtheorem{lemma}{Lemma}
\newtheorem{proposition}{Proposition}

\numberwithin{equation}{section}
\numberwithin{definition}{section}
\numberwithin{lemma}{section}
\numberwithin{theorem}{section}
\begin{document}

\author{Sezin Aykurt Sepet\\
\small K\i r\c{s}ehir Ahi Evran University, Department of Mathematics, K\i r\c{s}ehir, Turkey\\
\small saykurt@ahievran.edu.tr}
\title{Conformal Bi-slant Submersions}
\maketitle

\begin{abstract}
 We study conformal bi-slant submersions from almost Hermitian manifolds onto Riemannian manifolds as a generalized of conformal anti-invariant, conformal semi-invariant, conformal semi-slant, conformal slant and conformal hemi-slant submersions. We investigated the integrability of distributions and obtain necessary and sufficient conditions for the maps to have totally geodesic fibers. Also we studied the total geodesicity of such maps.\\


\textbf {2010 Subject Classification}: {53C15, 53C43}

\end{abstract}

\section{Introduction}

In complex geometry,  as a generalization of holomorphic and totally real immersions, slant immersions were defined by Chen \cite{chen}. Cabrerizo et al \cite{cabrerizo} defined bi-slant submanifolds in almost contact metric manifolds. In \cite{uddin} Uddin et al.  studied warped product bi-slant immersions in Kaehler manifolds. They proved that there do not exist any warped product bi-slant submanifolds of Kaehler manifolds other than hemi-slant warped products and CR-warped products.

The theory of Riemannian submersions as an analogue of isemetric immersions was initiated by O'Neill \cite{oneill} and Gray\cite{gray}. The Riemannian submersions are important in physics owing to applications in the Yang-Mills theory, Kaluza-Klein theory, robotic theory, supergravity and superstring theories. In Kaluza-Klein theory, the general solution of a recent model is given in point of harmonic maps satisfying Einstein equations (see \cite{yang1,kaluza 1,kaluza2,kaluza3,yang2,supergravity1,supergravity2}). Altafini \cite{altafini} expressed some applications of submersions in the theory of robotics and \c{S}ahin \cite{sahin} also investigated some applications of Riemannian submersions on redundant robotic chains. On the other hand Riemannian submersions are very useful in studying the geometry of Riemannian manifolds equipped with differentiable structures.
In \cite{watson} Watson introduced the notion of almost Hermitian submersions between almost complex manifolds. He investigated some geometric properties between base manifold and total manifold as well as fibers. \c{S}ahin \cite{sahinanti} introduced anti-invariant Riemannian submersions from almost Hermitian manifolds. He showed that such maps have some geometric properties. Also he studied slant submersions from almost Hermitian manifolds onto a Riemannian manifolds \cite{sahinslant}. Recently, considering different conditions on Riemannian submersions many studies have been done (see \cite{sezin,sayar2,sayar,sayar3,sahinsemi,hakan,Tastan2}).

As a special horizontally conformal maps which were introduced independently by Fuglede and Ishihara, horizontally conformal submersions are defined as follows $(M_{1},g_{1})$ and $(M_{2},g_{2})$ are Riemannian manifolds of dimension $m_{1}$ and $m_{2}$, respectively. A smooth submersion $f:(M_{1},g_{1})\rightarrow (M_{2},g_{2})$ is called a horizontally conformal submersion if there is a positive function $\lambda$ such that
\begin{align*}
\lambda^{2}g_{1}\left(X_{1},Y_{1}\right)=g_{2}\left(f_{*}X_{1},f_{*}Y_{1}\right)
\end{align*}
for all $X_{1}, Y_{1}\in \Gamma\left(\left(\ker f_{*}\right)^{\perp}\right)$.
Here a horizontally conformal submersion $f$ is called horizontally homothetic if the $grad \lambda$ is vertical i.e.
\begin{align*}
\mathcal{H}\left(grad\lambda\right)=0.
\end{align*}
 We denote by $\mathcal{V}$ and $\mathcal{H}$ the projections on the vertical distributions $\left(ker f_{*}\right)$ and horizontal distributions $\left(ker f_{*}\right)^{\perp}$. It can be said that Riemannian submersion is a special horizontally conformal submersion with $\lambda=1$.
Recently, Akyol and \c{S}ahin introduced conformal anti-invariant submersions \cite{akyolantiinv2}, conformal semi-invariant submersion\cite{akyolsemiinv}, conformal slant submersion \cite{akyol1} and conformal semi-slant submersions\cite{akyolsemislant}. Also the geometry of conformal submersions have been studied by several authors \cite{gunduzalpconf,kumar}.\\
In section 2 we review basic formulas and definitions needed for this paper. In section 3, we define the new conformal bi-slant submersion from almost Hermitian manifolds onto Riemannian manifolds and present a example. We investigate the geometry of the horizontal distribution and the vertical distribution. Finally we obtain necessary and sufficient conditions for a conformal bi-slant submersion to be totally geodesic.

\section{Preliminaries}

Let $\left(M_{1},g_{1},J\right)$ be an almost Hermitian manifold. Then this means that $M_{1}$ admits a tensor field $J$ of type $(1,1)$ on $M_{1}$ which satisfy
\begin{align}
J^{2}=-I, \ \ g_{1}\left(JE_{1},JE_{2}\right)=g_{1}\left(E_{1},E_{2}\right)
\end{align}
for $E_{1},E_{2}\in \Gamma(TM_{1})$. An almost Hermitian manifold $M_{1}$ is called Kaehlerian manifold if
\begin{align*}
\left(\nabla_{E_{1}}J\right)E_{2}=0, \ \ E_{1},E_{2}\in\Gamma\left(TM_{1}\right)
\end{align*}
where $\nabla$ is the operator of Levi-Civita covariant differentiation.

Now, we will give some definitions and theorems about the concept of (horizontally) conformal submersions.
\begin{definition}
	Let $(M_{1},g_{1})$ and $(M_{2},g_{2})$ are two Riemannian manifolds with the dimension $m_{1}$ and $m_{2}$, respectively. A smooth map $f:(M_{1},g_{1})\rightarrow (M_{2},g_{2})$ is called horizontally weakly conformal or semi conformal at $q\in M$ if, either
	\begin{enumerate}[i.]
		\item $df_{q}=0$, or
		\item $df_{q}$ is surjective and there exists a number $\Omega(q)\neq 0$ satisfying
		\begin{align*}
		g_{2}\left(df_{q}X,df_{q}Y\right)=\Omega(q)g_{1}\left(X,Y\right)
		\end{align*}
		for $X,Y\in \Gamma\left(\ker(df)\right)^{\perp}$.
	\end{enumerate}
Here the number $\Omega(q)$ is called the square dilation. Its square root $\lambda(q)=\sqrt{\Omega(q)}$ is called the dilation. The map $f$ is called horizontally weakly conformal or semi-conformal on $M_{1}$ if it is horizontally weakly conformal at every point of $M_{1}$. it is said to be a conformal submersion if $f$ has no critical point.
\end{definition}
Let $f:M_{1}\rightarrow M_{2}$ be a submersion. A vector field $X_{1}$ on $M_{1}$ is called a basic vector field if $X_{1}\in \Gamma\left(\left(\ker f_{*}\right)^{\perp}\right)$ and $f$-related with a vector field $X_{2}$ on $M_{2}$ i.e $f_{*}(X_{1q})=X_{2f(q)}$ for $q\in M_{1}$.

The two $(1,2)$ tensor fields $\mathcal{T}$ and $\mathcal{A}$ on $M$ are given by the formulas
\setlength\arraycolsep{2pt}
\begin{eqnarray}
\mathcal{T}(E_{1},E_{2})&=&\mathcal{T}_{E_{1}}E_{2}=\mathcal{H}\nabla_{\mathcal{V}E_{1}}\mathcal{V}E_{2}+\mathcal{V}\nabla_{\mathcal{V}E_{1}}\mathcal{H}E_{2} \label{2.2} \\
\mathcal{A}(E_{1},E_{2})&=&\mathcal{A}_{E_{1}}E_{2}=\mathcal{V}\nabla_{\mathcal{H}E_{1}}\mathcal{H}E_{2}+\mathcal{H}\nabla_{\mathcal{H}E_{1}}\mathcal{V}E_{2} \label{2.3}
\end{eqnarray}
for $E_{1},E_{2}\in \Gamma\left(TM\right)$ \cite{falcitelli}.\\

Note that a Riemannian submersion $f:M_{1}\longrightarrow M_{2}$ has totally geodesic fibers if and only if $\mathcal{T}$ vanishes identically.

Considering the equations (2.3) and (2.4), one can write
\setlength\arraycolsep{2pt}
\begin{eqnarray}
\nabla_{U_{1}}U_{2}&=&\mathcal{T}_{U_{1}}U_{2}+\bar{\nabla}_{U_{1}}U_{2} \label{2.4}\\
\nabla_{U_{1}}X_{1}&=&\mathcal{H}\nabla_{U_{1}}X_{1}+\mathcal{T}_{U_{1}}X_{1} \label{2.5}\\
\nabla_{X_{1}}U_{1}&=&\mathcal{A}_{X_{1}}U_{1}+\mathcal{V}\nabla_{X_{1}}U_{1} \label{2.6}\\
\nabla_{X_{1}}X_{2}&=&\mathcal{H}\nabla_{X_{1}}X_{2}+\mathcal{A}_{X_{1}}X_{2} \label{2.7}
\end{eqnarray}
for $X_{1},X_{2}\in\Gamma\left(\left(\ker f_{*}\right)^{\perp}\right)$ and $U_{1},U_{2}\in\Gamma\left(\ker f_{*}\right)$, where $\bar{\nabla}_{U_{1}}U_{2}=\mathcal{V}\nabla_{U_{1}}U_{2}$. Then we easily seen that $\mathcal{T}_{U_{1}}$ and $\mathcal{A}_{X_{1}}$ are skew-symmetric i.e $
g_{1}\left(\mathcal{A}_{X_{1}}E_{1},E_{2}\right)=-g_{1}\left(E_{1},\mathcal{A}_{X_{1}}E_{2}\right)$ and
$g_{1}\left(\mathcal{T}_{U_{1}}E_{1},E_{2}\right)=-g_{1}\left(E_{1},\mathcal{T}_{U_{1}}E_{2}\right)$ for any $E_{1},E_{2}\in\Gamma\left(TM_{1}\right)$. For the special case  where $f$ as the horizontal, the following Proposition be given:
\begin{proposition}
	Let $f:\left(M_{1},g_{1}\right)\rightarrow\left(M_{2},g_{2}\right)$ be a horizontally conformal submersion with dilation $\lambda$ and $X_{1},X_{2}\in\Gamma\left(\left(\ker f_{*}\right)^{\perp}\right)$, then
	\begin{align}
	\mathcal{A}_{X_{1}}X_{2}=\frac{1}{2}\left(\mathcal{V}\left[X_{1},X_{2}\right]-\lambda^{2}g_{1}\left(X_{1},X_{2}\right)grad_{\mathcal{V}}\left(\frac{1}{\lambda^{2}}\right)\right)
	\end{align}
 \end{proposition}

Let $f:\left(M_{1},g_{1}\right)\rightarrow\left(M_{2},g_{2}\right)$ be a smooth map between $\left(M_{1},g_{1}\right)$ and $\left(M_{2},g_{2}\right)$  Riemannian manifolds. Then the second fundamental form of $f$ is given by
\begin{align}\label{2.9}
\left(\nabla f_{*}\right)\left(E_{1},E_{2}\right)=\nabla^{f}_{E_{1}}f_{*}(E_{2})-f_{*}\left(\bar{\nabla}_{E_{1}}E_{2}\right)
\end{align}
for any $E_{1},E_{2}\in\Gamma\left(TM_{1}\right)$. It is known that the second fundamental form $f$ is symmetric \cite{baird}.
\begin{lemma}
	Suppose that $f:M_{1}\rightarrow M_{2}$ is a horizontally conformal submersion. Then for $X_{1},X_{2}\in \Gamma\left(\left(\ker f_{*}\right)^{\perp}\right)$ and $U_{1},U_{2}\in \Gamma\left(\ker f_{*}\right)$ we have
	\begin{enumerate}[i.]
		\item $\left(\nabla f_{*}\right)\left(X_{1},X_{2}\right)=X_{1}\left(\ln\lambda\right)f_{*}X_{2}+X_{2}\left(\ln\lambda\right)f_{*}X_{1}-g_{1}\left(X_{1},X_{2}\right)f_{*}\left(\nabla\ln\lambda\right)$
		\item $\left(\nabla f_{*}\right)\left(U_{1},U_{2}\right)=-f_{*}\left(\mathcal{T}_{U_{1}}U_{2}\right)$
		\item $\left(\nabla f_{*}\right)\left(X_{1},U_{1}\right)=-f_{*}\left(\bar{\nabla}_{X_{1}}U_{1}\right)=-f_{*}\left(\mathcal{A}_{X_{1}}V_{1}\right)$.
	\end{enumerate}
\end{lemma}
 The smoooth map  $f$ is called a totally geodesic map if $\left(\nabla f_{*}\right)\left(E_{1},E_{2}\right)=0$ for $E_{1},E_{2}\in\Gamma(TM)$ \cite{baird}.
 
We assume that $g$ is a Riemannian metric tensor on the manifold $M=M_{1}\times M_{2}$ and the canonical foliations $D_{M_{1}}$ and $D_{M_{2}}$ intersect vertically everywhere. Then $g$ is the metric tensor of a usual product of Riemannian manifold if and only if $D_{M_{1}}$ and $D_{M_{2}}$ are totally geodesic foliations.

\section{Conformal Bi-Slant Submersions}

\begin{definition}
	Let $\left(M_{1},g_{1},J\right)$ be an almost Hermitian manifold and $\left(M_{2},g_{2}\right)$ a Riemannian manifold. A horizontal conformal submersion $f:M_{1}\longrightarrow M_{2}$ is called a conformal bi-slant submersion if $D$ and $\bar{D}$ are slant distributions with the slant angles $\theta$ and $\bar{\theta}$, respectively, such that $\ker f_{*}=D\oplus \bar{D}$. $f$ is called proper if its slant angles satisfy $\theta,\bar{\theta}\neq 0,\frac{\pi}{2}$.
\end{definition}
 We now give a example of a proper conformal bi-slant submersion.
 \begin{example}
  We consider  the compatible almost complex structure $J_{\omega}$ on $\mathbb{R}^{8}$ such that
 \begin{align*}
 J_{\omega}=\left(\cos\omega\right) J_{1}+\left(\sin\omega\right) J_{2}, \ 0<\omega\leq\frac{\pi}{2}
 \end{align*}
 where
 \begin{align*}
 J_{1}\left(x_{1},x_{2},x_{3},x_{4},x_{5},x_{6},x_{7},x_{8}\right)=\left(-x_{2},x_{1},-x_{4},x_{3},-x_{6},x_{5},-x_{8},x_{7}\right)\\
 J_{2}\left(x_{1},x_{2},x_{3},x_{4},x_{5},x_{6},x_{7},x_{8}\right)=\left(-x_{3},x_{4},x_{1},-x_{2},-x_{7},x_{8},x_{5},-x_{6}\right)
 \end{align*}
 Consider a submersion $f:\mathbb{R}^{8} \rightarrow \mathbb{R}^{4}$ defined by
 \begin{align*}
 f\left(x_{1},x_{2},x_{3},x_{4},x_{5},x_{6},x_{7},x_{8}\right)=\pi^{5}\left(\frac{x_{1}-x_{3}}{\sqrt{2}},x_{4}, \frac{x_{5}-x_{6}}{\sqrt{2}},x_{7}\right)
 \end{align*}
 Then it follows that
 \begin{align*}
 D=span\{U_{1}=\frac{1}{\sqrt{2}}\left(\frac{\partial}{\partial x_{1}}+\frac{\partial}{\partial x_{3}}\right),U_{2}=\frac{\partial}{\partial x_{2}}\}\\
 \bar{D}=span\{U_{3}=\frac{1}{\sqrt{2}}\left(\frac{\partial}{\partial x_{5}}+\frac{\partial}{\partial x_{6}}\right),U_{4}=\frac{\partial}{\partial x_{8}}\}
 \end{align*}
 Thus $f$ is conformal bi-slant submersion with $\theta$ and $\bar{\theta}$ such that $\cos\theta=\frac{1}{\sqrt{2}}\cos\omega$ and $\cos\bar{\theta}=\frac{1}{\sqrt{2}}\sin\omega$.
 \end{example}

Suppose that $f$ is a conformal bi-slant submersion from a almost Hermitian manifold $\left(M_{1},g_{1},J_{1}\right)$ onto a Riemannian manifold $(M_{2},g_{2})$. For $U_{1}\in\Gamma\left(\ker f_{*}\right)$, we have
\begin{equation}\label{3.1}
U_{1}=\alpha U_{1}+\beta U_{1}
\end{equation}
where $\alpha U_{1}\in\Gamma\left(D_{1}\right)$ and $\beta U_{1}\in\Gamma\left(D_{2}\right)$.\\
Also, for $U_{1}\in\Gamma\left(\ker f_{*}\right)$, we write
\begin{equation}\label{3.2}
JU_{1}=\xi U_{1}+\eta U_{1}
\end{equation}
where $\xi U_{1}\in\Gamma\left(\ker f_{*}\right)$ and $\eta U_{1}\in\Gamma\left(\ker f_{*}\right)^\perp$.\\
For $X_{1}\in\Gamma\left(\left(\ker f_{*}\right)^\perp\right)$, we have
\begin{equation}\label{3.3}
JX_{1}=\mathcal{B}X_{1}+\mathcal{C}X_{1}
\end{equation}
where $\mathcal {B}X_{1}\in\Gamma\left(\ker f_{*}\right)$ and $\mathcal{C}X_{1}\in\Gamma\left(\left(\ker f_{*}\right)^\perp\right)$.\\
The horizontal distribution $(\ker f_{*})^{\perp}$ is decompesed as
\begin{align*}
(\ker f_{*})^{\perp}=\eta D_{1}\oplus\eta D_{2}\oplus\mu
\end{align*}
where $\mu$ is the complementary distribution to $\eta D_{1}\oplus\eta D_{2}$ in $(\ker f_{*})^{\perp}$.\\

Considering Definition 3.1 we can give the following result that we will use throughout the article.

\begin{theorem}
	Suppose that $f$ is a conformal bi-slant submersion from an almost Hermitian manifold $(M_{1},g_{1},J)$ onto a Riemannian manifold $(M_{2},g_{2})$. Then we have
	\begin{enumerate}[i)]
		\item $\xi^{2}U_{1}=-\left(\cos^{2}\theta\right)U_{1} \  \text{for} \  U_{1}\in\Gamma\left(D\right)$
		\item $\xi^{2}{V}_{1}=-\left(\cos^{2}\bar{\theta}\right)V_{1} \ \text{for} \ V_{1}\in\Gamma\left(\bar{D}\right)$
	\end{enumerate}

\end{theorem}
\begin{proof}
	The proof of this theorem is similar to slant immersions \cite{chen}.
\end{proof}

\begin{theorem}
	Suppose that $f$ is a proper conformal bi-slant submersion from a Kaehlerian manifold $(M_{1},g_{1},J)$ onto a Riemannian manifold $(M_{2},g_{2})$ with slant functions $\theta,\bar{\theta}$. Then
	\begin{enumerate}
		\item[\textit{i)}] the distribution $D$ is integrable if and only if
	\begin{align*}
\lambda^{-2}g_{2}\left(\nabla f_{*}\left(U_{1},\eta U_{2}\right),f_{*}\eta V_{1}\right)
=&g_{1}\left(\mathcal{T}_{U_{2}}\eta\xi U_{1}-\mathcal{T}_{U_{1}}\eta\xi U_{2},V_{1}\right)\\&+g_{1}\left(\mathcal{T}_{U_{1}}\eta U_{2}-\mathcal{T}_{U_{2}}\eta U_{1},\xi V_{1}\right)\\&+\lambda^{-2}g_{2}\left(\nabla f_{*}\left(U_{2},\eta U_{1}\right),f_{*}\eta V_{1}\right).
\end{align*}
		\item[\textit{ii)}] the distribution $\bar{D}$ is integrable if and only if
			\begin{align*}
		\lambda^{-2}g_{2}\left(\nabla f_{*}\left(V_{1},\eta V_{2}\right),f_{*}\eta U_{1}\right)
		=&g_{1}\left(\mathcal{T}_{V_{2}}\eta\xi V_{1}-\mathcal{T}_{V_{1}}\eta\xi V_{2},U_{1}\right)\\&+g_{1}\left(\mathcal{T}_{V_{1}}\eta V_{2}-\mathcal{T}_{V_{2}}\eta V_{1},\xi U_{1}\right)\\&+\lambda^{-2}g_{2}\left(\nabla f_{*}\left(V_{2},\eta V_{1}\right),f_{*}\eta U_{1}\right).
		\end{align*}
	\end{enumerate}
	where $U_{1},U_{2}\in \Gamma\left(D\right)$, $V_{1},V_{2}\in \Gamma\left(\bar{D}\right)$.
\end{theorem}

\begin{proof}	
$i)$ From $U_{1},U_{2} \in \Gamma\left(D\right)$ and $V_{1} \in \Gamma\left(\bar{D}\right)$ we have
	\setlength\arraycolsep{2pt}
	\begin{align*}
	g_{1}\left([U_{1},U_{2}],V_{1}\right)=&g_{1}\left(\nabla_{U_{1}}\xi U_{2},J V_{1}\right)+g_{1}\left(\nabla_{U_{1}}\eta U_{2},J V_{1}\right)\\&-g_{1}\left(\nabla_{U_{2}}\xi U_{1},J V_{1}\right)-g_{1}\left(\nabla_{U_{2}}\eta U_{1},J V_{1}\right).
	\end{align*}
	Considering Theorem 3.1 we arrive
	\begin{align*}
	\sin^{2}\theta g_{1}\left([U_{1},U_{2}],V_{1}\right)
	=&-g_{1}\left(\nabla_{U_{1}}\eta\xi U_{2},V_{1}\right)+g_{1}\left(\nabla_{U_{1}}\eta U_{2}, JV_{1}\right)\\&+g_{1}\left(\nabla_{U_{2}}\eta\xi U_{1},V_{1}\right)-g_{1}\left(\nabla_{U_{2}}\eta U_{1},JV_{1}\right).
	\end{align*}
	By using the equation \eqref{2.5} we obtain
	\begin{align*}
	\sin^{2}\theta g_{1}\left([U_{1},U_{2}],V_{1}\right)
	=&g_{1}\left(\mathcal{T}_{U_{2}}\eta\xi U_{1}-\mathcal{T}_{U_{1}}\eta\xi U_{2},V_{1}\right)+g_{1}\left(\mathcal{T}_{U_{1}}\eta U_{2}-\mathcal{T}_{U_{2}}\eta U_{1},\xi V_{1}\right)\\&-\lambda^{-2}g_{2}\left(\nabla f_{*}\left(U_{1},\eta U_{2}\right),f_{*}\eta V_{1}\right)\\&+\lambda^{-2}g_{2}\left(\nabla f_{*}\left(U_{2},\eta U_{1}\right),f_{*}\eta V_{1}\right).
	\end{align*}
	The proof of $ii)$ can be made by applying similar calculations.
\end{proof}

\begin{theorem}
	Suppose that $f$ is a proper conformal bi-slant submersion from a Kaehlerian manifold $(M_{1},g_{1},J)$ onto a Riemannian manifold $(M_{2},g_{2})$ with slant functions $\theta,\bar{\theta}$. Then the distribution $D$ defines a totally geodesic foliation if and only if
	\small{
	\begin{align}
\lambda^{-2}g_{2}\left(\nabla f_{*}\left(\eta U_{2},U_{1}\right),f_{*}\eta V_{1}\right)=-g_{1}\left(\mathcal{T}_{U_{1}}\eta\xi U_{2},V_{1}\right)+g_{1}\left(\mathcal{T}_{U_{1}}\eta U_{2},\xi V_{1}\right).
\end{align}}\normalsize
	and
	\begin{align}
\lambda^{-2}g_{2}\left(\nabla^{f}_{X_{1}}f_{*}\eta U_{1},f_{*}\eta U_{2}\right)=&-\sin^{2}\theta g_{1}\left(\left[U_{1},X_{1}\right],U_{1}\right)+g_{1}\left(\mathcal{A}_{X_{1}}\eta\xi U_{1},U_{2}\right)\nonumber\\&+g_{1}\left(grad(\ln\lambda),X_{1}\right)g_{1}\left(\eta U_{1},\eta U_{2}\right)\nonumber\\&+g_{1}\left(grad(\ln\lambda),\eta U_{1}\right)g_{1}\left(X_{1},\eta U_{2}\right)\nonumber\\&+g_{1}\left(grad(\ln\lambda),\eta U_{2}\right)g_{1}\left(X_{1},\eta U_{1}\right)\nonumber\\&-g_{1}\left(\mathcal{A}_{X_{1}}\eta U_{1},\xi U_{2}\right)
	\end{align}
	where $U_{1},U_{2}\in \Gamma\left(D\right)$, $V_{1}\in \Gamma\left(\bar{D}\right)$ and $X_{1}\in\Gamma\left(\left(\ker f_{*}\right)^{\perp}\right)$.
\end{theorem}

\begin{proof}
	
	For $U_{1},U_{2}\in \Gamma\left(D\right)$ and $V_{1}\in\Gamma\left(\bar{D}\right)$ we have
	\begin{align*}
	g_{1}\left(\nabla_{U_{1}}U_{2},V_{1}\right)=&-g_{1}\left(\nabla_{U_{1}}\xi^{2}U_{2},V_{1}\right)-g_{1}\left(\nabla_{U_{1}}\eta\xi U_{2},V_{1}\right)+g_{1}\left(\nabla_{U_{1}}\eta U_{2},JV_{1}\right).
	\end{align*}
	Thus we can write
	\begin{align*}
	\sin^{2}\theta g_{1}\left(\nabla_{U_{1}}U_{2},V_{1}\right)=&-g_{1}\left(\mathcal{T}_{U_{1}}\eta\xi U_{2},V_{1}\right)+g_{1}\left(\mathcal{T}_{U_{1}}\eta U_{2},\xi V_{1}\right)\\&+g_{1}\left(\mathcal{H}\nabla_{U_{1}}\eta U_{2},\eta V_{1}\right).
	\end{align*}
	Using \eqref{2.9} we obtain 
	\begin{align*}
	\sin^{2}\theta g_{1}\left(\nabla_{U_{1}}U_{2},V_{1}\right)=&-g_{1}\left(\mathcal{T}_{U_{1}}\eta\xi U_{2},V_{1}\right)+g_{1}\left(\mathcal{T}_{U_{1}}\eta U_{2},\xi V_{1}\right)\\&-\lambda^{-2}g_{2}\left(\nabla f_{*}\left(\eta U_{2},U_{1}\right),f_{*}\eta V_{1}\right).
	\end{align*}
	which is first equation in Theorem 3.3.
	
	On the other hand any $U_{1},U_{2} \in \Gamma(D) $ and $X_{1}\in\Gamma\left(\left(\ker f_{*}\right)^{\perp}\right)$ we can write
	\setlength\arraycolsep{2pt}
	\begin{align*}
	g_{1}\left(\nabla_{U_{1}}U_{2},X_{1}\right)=&-g_{1}\left(\left[U_{1},X_{1}\right],U_{2}\right)-g_{1}\left(\nabla_{X_{1}}U_{1},U_{2}\right)\\
	=&-g_{1}\left(\left[U_{1},X_{1}\right],U_{2}\right)+g_{1}\left(\nabla_{X}J\xi U_{1},U_{2}\right)-g_{1}\left(\nabla_{X_{1}}\eta U_{1},J U_{2}\right).
	\end{align*}
	Using Theorem 3.1, we arrive following equation
	\begin{align*}
	g_{1}\left(\nabla_{U_{1}}U_{2},X_{1}\right)=&-g_{1}\left(\left[U_{1},X_{1}\right],U_{2}\right)-\cos^{2}\theta g_{1}\left(\nabla_{X_{1}}U_{1},U_{2}\right)\\&+g_{1}\left(\nabla_{X_{1}}\eta\xi U_{1},U_{2}\right)-g_{1}\left(\nabla_{X_{1}}\eta U_{1},JU_{2}\right)\nonumber
	\end{align*}
	From \eqref{2.7} and Lemma 2.1  we have
	\begin{align*}
	\sin^2\theta g_{1}\left(\nabla_{U_{1}}U_{2},X_{1}\right)=&-\sin^{2}\theta g_{1}\left(\left[U_{1},X_{1}\right],U_{1}\right)+g_{1}\left(\mathcal{A}_{X_{1}}\eta\xi U_{1},U_{2}\right)\\&-g_{1}\left(\mathcal{A}_{X_{1}}\eta U_{1},\xi U_{2}\right)-\lambda^{-2}g_{2}\left(\nabla^{f}_{X_{1}}f_{*}\eta U_{1},f_{*}\eta U_{2}\right)\\&+g_{1}\left(grad(\ln\lambda),X_{1}\right)g_{1}\left(\eta U_{1},\eta U_{2}\right)\\&+g_{1}\left(grad(\ln\lambda),\eta U_{1}\right)g_{1}\left(X_{1},\eta U_{2}\right)\\&+g_{1}\left(grad(\ln\lambda),\eta U_{2}\right)g_{1}\left(X_{1},\eta U_{1}\right)\nonumber
	\end{align*}
	This completes the proof.
\end{proof}

\begin{theorem}
	Suppose that $f$ is a proper conformal bi-slant submersion from a Kaehlerian manifold $(M_{1},g_{1},J)$ onto a Riemannian manifold $(M_{2},g_{2})$ with slant functions $\theta,\bar{\theta}$. Then the distribution $\bar{D}$ defines a totally geodesic foliation if and only if
	\small{
	\begin{align}
	\lambda^{-2}g_{2}\left(\nabla f_{*}\left(\eta V_{2},V_{1}\right),f_{*}\eta U_{1}\right)=-g_{1}\left(\mathcal{T}_{V_{1}}\eta\xi V_{2},U_{1}\right)+g_{1}\left(\mathcal{T}_{V_{1}}\eta V_{2},\xi U_{1}\right).
	\end{align}}\normalsize
	and
	\begin{align}
	\lambda^{-2}g_{2}\left(\nabla^{f}_{X_{1}}f_{*}\eta V_{1},f_{*}\eta V_{2}\right)=&-\sin^{2}\bar{\theta} g_{1}\left(\left[V_{1},X_{1}\right],V_{1}\right)+g_{1}\left(\mathcal{A}_{X_{1}}\eta\xi V_{1},V_{2}\right)\nonumber\\&+g_{1}\left(grad(\ln\lambda),X_{1}\right)g_{1}\left(\eta V_{1},\eta V_{2}\right)\nonumber\\&+g_{1}\left(grad(\ln\lambda),\eta V_{1}\right)g_{1}\left(X_{1},\eta V_{2}\right)\nonumber\\&+g_{1}\left(grad(\ln\lambda),\eta V_{2}\right)g_{1}\left(X_{1},\eta V_{1}\right)\nonumber\\&-g_{1}\left(\mathcal{A}_{X_{1}}\eta V_{1},\xi V_{2}\right)
	\end{align}
	where $U_{1}\in \Gamma\left(D\right)$, $V_{1}, V_{2}\in \Gamma\left(\bar{D}\right)$ and $X_{1}\in\Gamma\left(\left(\ker f_{*}\right)^{\perp}\right)$.
\end{theorem}

\begin{proof}
	The proof of this theorem is similar to the proof of Theorem 3.3.
\end{proof}

\begin{theorem}
	Suppose that $f$ is a proper conformal bi-slant submersion from a Kaehlerian manifold $(M_{1},g_{1},J)$ onto a Riemannian manifold $(M_{2},g_{2})$ with slant functions $\theta,\bar{\theta}$.Then, the vertical distribution $\left(\ker f_{*}\right)$ is a locally product $M_{D}\times M_{\bar{D}}$ if and only if the equations (3.4), (3.5), (3.6) and (3.7) are hold where $M_{D}$ and $M_{\bar{D}}$ are integral manifolds of the distributions $D$ and $\bar{D}$, respectively. 	
\end{theorem}

\begin{theorem}
Suppose that $f$ is a proper conformal bi-slant submersion from a Kaehlerian manifold $(M_{1},g_{1},J)$ onto a Riemannian manifold $(M_{2},g_{2})$ with slant functions $\theta,\bar{\theta}$. Then the distribution $\left(\ker f_{*}\right)^\perp$ defines a totally geodesic foliation if and only if
\begin{align}
\lambda^{-2}g_{2}\left(\nabla^{f}_{X_{1}}f_{*}\eta U_{1},f_{*}CX_{2}\right)=&-g_{1}\left(\mathcal{A}_{X_{1}}\eta U_{1},BX_{2}\right)
+\lambda^{-2}g_{2}\left(\nabla^{f}_{X_{1}}f_{*}\eta\xi U_{1},f_{*}X_{2}\right)\nonumber\\&-g_{1}\left(grad\ln\lambda,X_{1}\right)g_{1}\left(\eta\xi U_{1},X_{2}\right)\nonumber\\&-g_{1}\left(grad\ln\lambda,\eta\xi U_{1}\right)g_{1}\left(X_{1},X_{2}\right)\nonumber\\&+g_{1}\left(X_{1},\eta\xi U_{1}\right)g_{1}\left(grad\ln\lambda,X_{2}\right)\nonumber\\&+g_{1}\left(grad\ln\lambda,\eta U_{1}\right)g_{1}\left(X_{1},CX_{2}\right)\nonumber\\&-g_{1}\left(X_{1},\eta U_{1}\right)g_{1}\left(grad\ln\lambda,CX_{2}\right).
\end{align}
and
\begin{align}
\lambda^{-2}g_{2}\left(\nabla^{f}_{X_{1}}f_{*}\eta V_{1},f_{*}CX_{2}\right)=&-g_{1}\left(\mathcal{A}_{X_{1}}\eta V_{1},BX_{2}\right)
+\lambda^{-2}g_{2}\left(\nabla^{f}_{X_{1}}f_{*}\eta\xi V_{1},f_{*}X_{2}\right)\nonumber\\&-g_{1}\left(grad\ln\lambda,X_{1}\right)g_{1}\left(\eta\xi V_{1},X_{2}\right)\nonumber\\&-g_{1}\left(grad\ln\lambda,\eta\xi V_{1}\right)g_{1}\left(X_{1},X_{2}\right)\nonumber\\&+g_{1}\left(X_{1},\eta\xi V_{1}\right)g_{1}\left(grad\ln\lambda,X_{2}\right)\nonumber\\&+g_{1}\left(grad\ln\lambda,\eta V_{1}\right)g_{1}\left(X_{1},CX_{2}\right)\nonumber\\&-g_{1}\left(X_{1},\eta V_{1}\right)g_{1}\left(grad\ln\lambda,CX_{2}\right).
\end{align}
	where $X_{1}, X_{2} \in\Gamma\left(\ker f_{*}\right)^\perp$, $U_{1}\in\Gamma\left(D\right)$ and $V_{1}\in\Gamma\left(\bar{D}\right)$.
\end{theorem}

\begin{proof}
	For $X_{1},X_{2} \in\Gamma\left(\ker \pi_{*}\right)^\perp$ and $U_{1}\in\Gamma\left(D\right)$ we can write
	\setlength\arraycolsep{2pt}
	\begin{eqnarray*}
	g_{1}\left(\nabla_{X_{1}}X_{2},U_{1}\right)=-g_{1}\left(\nabla_{X_{1}}\xi U_{1}, JX_{2}\right)-g_{1}\left(\nabla_{X_{1}}\eta U_{1},JX_{2}\right)
	\end{eqnarray*}
	From Theorem 3.1 we have
	\begin{align*}
	g_{1}\left(\nabla_{X_{1}}X_{2},U_{1}\right)=&-\cos^{2}\theta g_{1}\left(\nabla_{X_{1}}U_{1},X_{2}\right)+g_{1}\left(\nabla_{X_{1}}\eta\xi U_{1},X_{2}\right)\\&-g_{1}\left(\nabla_{X_{1}}\eta U_{1},JX_{2}\right)
	\end{align*}
	By using the equation (2.7) we derive
	\begin{align*}
	\sin^{2}\theta g\left(\nabla_{X_{1}}X_{2},U_{1}\right)=&g\left(\mathcal{H}\nabla_{X_{1}}\eta\xi U_{1},X_{2}\right)
	-g\left(\mathcal{H}\nabla_{X_{1}}\eta U_{1},CX_{2}\right)\\&-g\left(\nabla_{X_{1}}\eta U_{1},BX_{2}\right).
	\end{align*}
	Then it follows from Lemma 2.1 that
	\begin{align*}
	\sin^{2}\theta g_{1}\left(\nabla_{X_{1}}X_{2},U_{1}\right)=&-g_{1}\left(\mathcal{A}_{X_{1}}\eta U_{1},BX_{2}\right)
	+\lambda^{-2}g_{2}\left(\nabla^{f}_{X_{1}}f_{*}\eta\xi U_{1},f_{*}X_{2}\right)\\&-g_{1}\left(grad\ln\lambda,X_{1}\right)g_{1}\left(\eta\xi U_{1},X_{2}\right)\\&-g_{1}\left(grad\ln\lambda,\eta\xi U_{1}\right)g_{1}\left(X_{1},X_{2}\right)\\&+g_{1}\left(X_{1},\eta\xi U_{1}\right)g_{1}\left(grad\ln\lambda,X_{2}\right)\\&-\lambda^{-2}g_{2}\left(\nabla^{f}_{X_{1}}f_{*}\eta U_{1},f_{*}CX_{2}\right)\\&+g_{1}\left(grad\ln\lambda,\eta U_{1}\right)g_{1}\left(X_{1},CX_{2}\right)\\&-g_{1}\left(X_{1},\eta U_{1}\right)g_{1}\left(grad\ln\lambda,CX_{2}\right).
	\end{align*}
	Thus we have the first desired equation.
	Similarly for  $X_{1},X_{2} \in\Gamma\left(\left(\ker \pi_{*}\right)^\perp\right)$ and $V_{1}\in\left(\bar{D}\right)$ we find
	\begin{align*}
	\sin^{2}\bar{\theta} g_{1}\left(\nabla_{X_{1}}X_{2},V_{1}\right)=&-g_{1}\left(\mathcal{A}_{X_{1}}\eta V_{1},BX_{2}\right)
+\lambda^{-2}g_{2}\left(\nabla^{f}_{X_{1}}f_{*}\eta\xi V_{1},f_{*}X_{2}\right)\\&-g_{1}\left(grad\ln\lambda,X_{1}\right)g_{1}\left(\eta\xi V_{1},X_{2}\right)\\&-g_{1}\left(grad\ln\lambda,\eta\xi V_{1}\right)g_{1}\left(X_{1},X_{2}\right)\\&+g_{1}\left(X_{1},\eta\xi V_{1}\right)g_{1}\left(grad\ln\lambda,X_{2}\right)\\&-\lambda^{-2}g_{2}\left(\nabla^{f}_{X_{1}}f_{*}\eta V_{1},f_{*}CX_{2}\right)\\&+g_{1}\left(grad\ln\lambda,\eta V_{1}\right)g_{1}\left(X_{1},CX_{2}\right)\\&-g_{1}\left(X_{1},\eta V_{1}\right)g_{1}\left(grad\ln\lambda,CX_{2}\right).
	\end{align*}
	Hence the proof is completed.
\end{proof}

\begin{theorem}
	Suppose that $f$ is a proper conformal bi-slant submersion from a Kaehlerian manifold $(M_{1},g_{1},J)$ onto a Riemannian manifold $(M_{2},g_{2})$ with slant functions $\theta,\bar{\theta}$.Then the distribution $\left(\ker f_{*}\right)$ defines a totally geodesic foliation on $M_{1}
	$ if and only if
		\setlength\arraycolsep{2pt}
		\small{
	\begin{align}
\lambda^{-2}g_{2}\left(\nabla^{f}_{X_{1}}f_{*}\omega U_{1},f_{*}\omega V_{1}\right)=&\left(\cos^{2}\theta-\cos^{2}\bar{\theta}\right)g_{1}\left(\nabla_{X_{1}}QU_{1},V_{1}\right)-g_{1}\left(\mathcal{A}_{X_{1}}V_{1},\eta\xi U_{1}\right)\nonumber\\&-g_{1}\left(\mathcal{A}_{X_{1}}\xi V_{1},\eta U_{1}\right)-\sin^{2}\theta g_{1}\left(\left[U_{1},X_{1}\right],V_{1}\right)\nonumber\\&-g_{1}\left(X_{1},\eta U_{1}\right)g_{1}\left(grad\ln\lambda,\eta V_{1}\right)\nonumber\\&+g_{1}\left(grad\ln\lambda,X_{1}\right)g_{1}\left(\eta U_{1},\eta V_{1}\right)\nonumber\\&+g_{1}\left(grad\ln\lambda,\eta U_{1}\right)g_{1}\left(X_{1},\eta V_{1}\right)
\end{align}}\normalsize
	where $X_{1}\in\Gamma\left(\left(\ker f_{*}\right)^\perp\right)$ and $U_{1},V_{1}\in\Gamma\left(\ker f_{*}\right)$.
\end{theorem}

\begin{proof}
	Given $X_{1} \in\Gamma\left(\left(\ker f_{*}\right)^\perp\right)$ and $U_{1},V_{1}\in\left(\ker f_{*}\right)$. Then we obtain
	{\small
		\begin{align*}
		g_{1}\left(\nabla_{U_{1}}V_{1},X_{1}\right)=&-g_{1}\left(\left[U_{1},X_{1}\right],V_{1}\right)+g_{1}\left(J\nabla_{X_{1}}\xi U_{1},V_{1}\right)-g_{1}\left(\nabla_{X_{1}}\eta U_{1},JV_{1}\right)
		\end{align*}}
	By using Theorem 3.1 we have
	\begin{align*}
	g_{1}\left(\nabla_{U_{1}}V_{1},X_{1}\right)=&-g_{1}\left(\left[U_{1},X_{1}\right],V_{1}\right)-\cos^{2}\theta g_{1}\left(\nabla_{X_{1}} PU_{1},V_{1}\right)\\&-\cos^{2}\bar{\theta}g_{1}\left(\nabla_{X} QU_{1},V_{1}\right)+g_{1}\left(\nabla_{X_{1}}\eta\xi U_{1},V_{1}\right)\\&-g_{1}\left(\nabla_{X_{1}}\omega U_{1},\xi V_{1}\right)-g_{1}\left(\nabla_{X_{1}}\omega U_{1},\eta V_{1}\right).
	\end{align*}
    Then we arrive
	\begin{align*}
	\sin^{2}\theta g_{1}\left(\nabla_{U_{1}}V_{1},X_{1}\right)=&\left(\cos^{2}\theta-\cos^{2}\bar{\theta}\right)g_{1}\left(\nabla_{X_{1}}QU_{1},V_{1}\right)\\&+g_{1}\left(\nabla_{X_{1}}\eta\xi U_{1},V_{1}\right)-\sin^{2}\theta g_{1}\left(\left[U_{1},X_{1}\right],V_{1}\right)\\&-g_{1}\left(\nabla_{X_{1}}\eta U_{1},\xi V_{1}\right)-g_{1}\left(\nabla_{X_{1}}\eta U_{1},\eta V_{1}\right)
	\end{align*}
	From the equation \eqref{2.6} and Lemma 2.1 we obtain
		\begin{align*}
	\sin^{2}\theta g_{1}\left(\nabla_{U_{1}}V_{1},X_{1}\right)=&\left(\cos^{2}\theta-\cos^{2}\bar{\theta}\right)g_{1}\left(\nabla_{X_{1}}QU_{1},V_{1}\right)-g_{1}\left(\mathcal{A}_{X_{1}}V_{1},\eta\xi U_{1}\right)\\&-\sin^{2}\theta g_{1}\left(\left[U_{1},X_{1}\right],V_{1}\right)-g_{1}\left(\mathcal{A}_{X_{1}}\xi V_{1},\eta U_{1}\right)\\&+g_{1}\left(grad\ln\lambda,X_{1}\right)g_{1}\left(\eta U_{1},\eta V_{1}\right)\\&+g_{1}\left(grad\ln\lambda,\eta U_{1}\right)g_{1}\left(X_{1},\eta V_{1}\right)\\&-g_{1}\left(X_{1},\eta U_{1}\right)g_{1}\left(grad\ln\lambda,\eta V_{1}\right)\\&-\lambda^{-2}g_{2}\left(\nabla^{f}_{X_{1}}f_{*}\eta U_{1},f_{*}\eta V_{1}\right)
	\end{align*}
	Using above equation the desired equality is achieved.
\end{proof}

\begin{theorem}
	Suppose that $f$ is a proper conformal bi-slant submersion from a Kaehlerian manifold $(M_{1},g_{1},J)$ onto a Riemannian manifold $(M_{2},g_{2})$ with slant functions $\theta,\bar{\theta}$. Then, the total space $M_{1}$ is a locally product $M_{1D}\times M_{\bar{1D}}\times M_{1\left(\ker f_{*}\right)^{\perp}} $ if and only if the equations (3.4), (3.5), (3.6), (3.7), (3.8) and (3.9) are hold where $M_{1D}$, $M_{1\bar{D}}$ and $M_{1\left(\ker f_{*}\right)^{\perp}}$ are integral manifolds of the distributions $D$, $\bar{D}$ and $\left(\ker f_{*}\right)^{\perp}$, respectively. 	
\end{theorem}

\begin{theorem}
	Suppose that $f$ is a proper conformal bi-slant submersion from a Kaehlerian manifold $(M_{1},g_{1},J)$ onto a Riemannian manifold $(M_{2},g_{2})$ with slant functions $\theta,\bar{\theta}$. Then, the total space $M_{1}$ is a locally product $M_{1\ker f_{*}}\times M_{1\left(\ker f_{*}\right)^{\perp}} $ if and only if the equations (3.8), (3.9) and (3.10)are hold where $M_{1\ker f_{*}}$ and $M_{1\left(\ker f_{*}\right)^{\perp}}$ are integral manifolds of the distributions $\ker f_{*}$ and $\left(\ker f_{*}\right)^{\perp}$, respectively. 	
\end{theorem}

\begin{theorem}
	Suppose that $f$ is a proper conformal bi-slant submersion from a Kaehlerian manifold $(M_{1},g_{1},J)$ onto a Riemannian manifold $(M_{2},g_{2})$ with slant functions $\theta,\bar{\theta}$. Then $f$ is totally geodesic if and only if

\begin{align*}
-\lambda^{-2}g_{2}\left(\nabla^{f}_{\eta V_{1}}f_{*}\eta U_{1},f_{*}JCX_{1}\right)=&\left(\cos^{2}\theta-\cos^{2}\bar{\theta}\right) g_{1}\left(\mathcal{T}_{U_{1}}QV_{1},X_{1}\right)\\&+\lambda^{-2}g_{2}\left(\nabla f_{*}\left(\xi U_{1},\eta V_{1}\right),f_{*} JCX_{1}\right)\\&-g_{1}\left(\eta U_{1},\eta V_{1}\right)g_{1}\left(grad\ln\lambda,JCX_{1}\right)\\&+\lambda^{-2}g_{2}\left(\nabla f_{*}\left(U_{1},\eta\xi V_{1}\right),f_{*}X_{1}\right)\\&-g_{1}\left(\mathcal{T}_{U_{1}}\eta V_{1},BX_{1}\right)
\end{align*}
	and
	\begin{align*}
\lambda^{-2}g_{2}\left(\nabla_{X_{1}}^{f}f_{*}\eta U_{1},f_{*}CX_{2}\right)=&\left(\cos^{2}\theta-\cos^{2}\bar{\theta}\right)g_{1}\left(\mathcal{A}_{X_{1}}QU_{1},X_{2}\right)\\&+\lambda^{-2}g_{2}\left(\nabla_{X_{1}}^{f}f_{*}\eta \xi U_{1},f_{*}X_{2}\right)\\&-g_{1}\left(grad\ln\lambda,X_{1}\right)g_{1}\left(\eta\xi U_{1},X_{2}\right)\\&-g_{1}\left(grad\ln\lambda,\eta\xi U_{1}\right)g_{1}\left(X_{1},X_{2}\right)\\&+g_{1}\left(X_{1},\eta\xi U_{1}\right)g_{1}\left(grad\ln\lambda,X_{2}\right)\\&+g_{1}\left(grad\ln\lambda,\eta U_{1}\right)g_{1}\left(X_{1},CX_{2}\right)\\&-g_{1}\left(X_{1},\eta U_{1}\right)g_{1}\left(grad\ln\lambda,CX_{2}\right)\\&+g_{1}\left(\mathcal{A}_{X_{1}}BX_{2},\eta U\right).
	\end{align*}

	where $X_{1},X_{2}\in\Gamma\left(\left(\ker f_{*}\right)^\perp\right)$ and $U_{1},V_{1}\in\Gamma\left(\ker f_{*}\right)$.
\end{theorem}

\begin{proof}
Given $U_{1},V_{1}\in \Gamma\left(\ker f_{*}\right)$ and $X_{1}\in\Gamma\left(\left(\ker f_{*}\right)^{\perp}\right)$ Then we write
	\begin{equation*}
	\lambda^{-2}g_{2}\left(\nabla f_{*}(U_{1},V_{1}),f_{*}X\right)=-\lambda^{-2}g_{2}\left(f_{*}\nabla_{U_{1}}V_{1},f_{*}X\right).
	\end{equation*}
	From Theorem 3.1 we obtain
	\setlength\arraycolsep{2pt}
	\begin{align*}
\left(\sin^{2}\theta\right) \lambda^{-2}g_{2}\left(\nabla f_{*}\left(U_{1},V_{1}\right),f_{*}X_{1}\right)=&\left(\cos^{2}\theta-\cos^{2}\bar{\theta}\right) g_{1}\left(\nabla_{U_{1}}QV_{1},X_{1}\right)\\&-g_{1}\left(\nabla_{U_{1}}\eta V_{1},JX_{1}\right)+g_{1}\left(\nabla_{U_{1}}\eta\xi V_{1},X_{1}\right)
	\end{align*}
	Considering \eqref{2.4}, \eqref{2.5} and Lemma 2.1 we find 
	\begin{align*}
	\left(\sin^{2}\theta\right) \lambda^{-2}g_{2}\left(\nabla f_{*}\left(U_{1},V_{1}\right),f_{*}X_{1}\right)=&\left(\cos^{2}\theta-\cos^{2}\bar{\theta}\right) g_{1}\left(\mathcal{T}_{U_{1}}QV_{1},X_{1}\right)\\&+\lambda^{-2}g_{2}\left(\nabla f_{*}\left(\xi U_{1},\eta V_{1}\right),f_{*} JCX_{1}\right)\\&-g_{1}\left(\eta U_{1},\eta V_{1}\right)g_{1}\left(grad\ln\lambda,JCX_{1}\right)\\&+\lambda^{-2}g_{2}\left(\nabla^{f}_{\eta V_{1}}f_{*}\eta U_{1},f_{*}JCX_{1}\right)\\&+\lambda^{-2}g_{2}\left(\nabla f_{*}\left(U_{1},\eta\xi V_{1}\right),f_{*}X_{1}\right)\\&-g_{1}\left(\mathcal{T}_{U_{1}}\eta V_{1},BX_{1}\right).
	\end{align*}
	Therefore we obtain the first equation of Theorem 3.6.\\
	On the other hand, for $X_{1},X_{2}\in \Gamma\left(\left(\ker f_{*}\right)^{\perp}\right)$ and $U_{1}\in\Gamma\left(\ker f_{*}\right)$ we can write
	\begin{align*}
	\left(\sin^{2}\theta\right) \lambda^{-2} g_{2}\left(\nabla f_{*}\left(U_{1},X_{1}\right),f_{*}X_{2}\right)=&\left(\cos^{2}\theta-\cos^{2}\bar{\theta}\right)g_{1}\left(\nabla_{X_{1}}QU_{1},X_{2}\right)\\&+g_{1}\left(\nabla_{X_{1}}\eta U_{1},BX_{2}\right)-g_{1}\left(\nabla_{X_{1}}\eta U_{1},CX_{2}\right).
	\end{align*}
	By using the equation (2.6) and Lemma 2.1, we arrive
	\setlength\arraycolsep{2pt}
	\begin{align*}
	\left(\sin^{2}\theta\right) \lambda^{-2}g_{2}\left(\nabla f_{*}\left(U_{1},X_{1}\right),f_{*}X_{2}\right)=&\left(\cos^{2}\theta-\cos^{2}\bar{\theta}\right)g_{1}\left(\mathcal{A}_{X_{1}}QU_{1},X_{2}\right)\\&+\lambda^{-2}g_{2}\left(\nabla_{X_{1}}^{f}f_{*}\eta \xi U_{1},f_{*}X_{2}\right)\\&-g_{1}\left(grad\ln\lambda,X_{1}\right)g_{1}\left(\eta\xi U_{1},X_{2}\right)\\&-g_{1}\left(grad\ln\lambda,\eta\xi U_{1}\right)g_{1}\left(X_{1},X_{2}\right)\\&+g_{1}\left(X_{1},\eta\xi U_{1}\right)g_{1}\left(grad\ln\lambda,X_{2}\right)\\&-\lambda^{-2}g_{2}\left(\nabla_{X_{1}}^{f}f_{*}\eta U_{1},f_{*}CX_{2}\right)\\&+g_{1}\left(grad\ln\lambda,\eta U_{1}\right)g_{1}\left(X_{1},CX_{2}\right)\\&-g_{1}\left(X_{1},\eta U_{1}\right)g_{1}\left(grad\ln\lambda,CX_{2}\right)\\&+g_{1}\left(\mathcal{A}_{X_{1}}BX_{2},\eta U_{1}\right).
	\end{align*}
	This concludes the proof.
\end{proof}

\end{document}